\documentclass[12pt]{article}
\usepackage{authblk}

\usepackage{graphicx} 
\usepackage{amsmath}
\usepackage{amsfonts}
\usepackage{amssymb}
\usepackage{amsthm}
\usepackage{mathrsfs}
\usepackage{comment}
\usepackage{orcidlink}

\newcommand\blfootnote[1]{
    \begingroup
    \renewcommand\thefootnote{}\footnote{#1}
    \addtocounter{footnote}{-1}
    \endgroup
}

\newtheorem{thm}{Theorem}
\newtheorem{lemma}{Lemma}
\newtheorem{prop}{Proposition}
\newtheorem{corollary}{Corollary}

\newtheorem{hyp}{Hypothesis}

\title{Playing with additivity conditions in multiplicative functions }
\author{Crystel Bujold\, \orcidlink{0000-0002-7094-2568}}
\author{Isabelle Taylor-Daoust}
\affil{Collège Bois-de-Boulogne}
\date{}

\begin{document}

\maketitle
\begin{abstract}
    Let $f:\mathbb{Z}\to\mathbb{C}$ be a multiplicative function. Assume there exist integers $a>1$ and $d>1$ such that $(a,d)=1$ and let  $\mathcal{P}_{a,d}=\{a+kd:k\in\mathbb{Z}\}$. Under mild extra conditions on $d$ and $f(a)$, we prove that $f(n)=n\chi(n)$ for all $n$ outside an explicit exceptional set depending on $d$, and some Dirichlet character $\chi$. 
    
\noindent \textbf{Keywords: }Multiplicative functions;  Additivity; Arithmetic progressions; Dirichlet characters.
\end{abstract}

\section{Introduction}

\blfootnote{The first author is supported by FRQNT grant for college research.}
\blfootnote{\textbf{Subjects: }11A99, 11N99, 11B25, 11A25 }

We want to play a game. Knowing that one of the things that makes integers interesting is the dichotomy between their additive structure and their multiplicative structure, we wish to play by applying variations of these two conditions to arithmetic functions and observe the consequences. As it turns out, just like we expect it to be for the integers, satisfying both additivity and multiplicativity conditions is very restrictive and the functions must be very close to the identity function. 
    
Although the question has been approached from different points of view, previous research on the subject has mostly revolved around ideas similar to the ones put forth in Spiro's paper \cite{SPIRO1992232}, in which she shows that if $f$ is a multiplicative function and $f(p+q)=f(p)+f(q)$ for all primes $p$ and $q$, then $f$ must be the identity function. 

Several variations of the problem have been considered, mostly by modifying the sets on which the additive property holds, but also by considering different functional equations besides simple additivity and generalizing to the analogous k-terms problems. (see for example \cite{chen2010f}, \cite{chung1996multiplicative}, \cite{de1997new}, \cite{hasanalizade2021multiplicative}, \cite{indlekofer2006additive}, \cite{park2018multiplicative},\cite{park2020additive}, \cite{phong2006characterization}, \cite{park2023multiplicativefunctionsadditivepositive}, \cite{you2016characterization}).

In this paper, we consider the case of a multiplicative function $f$ endowed with the additive property $f(u+v) = f(u)+f(v)$ when $u$ and $v$ are in a given arithmetic progression. The goal of the game is to show that having both multiplicative and additive properties is very restrictive and to give a characterization of such functions. In fact, as shown in many of the papers cited above, the conditions imposed on the functions mean they have to behave closely to the identity function, but in our case, the additivity imposed on an arithmetic progression will produce a twist by a Dirichlet character. More precisely, using essentially elementary techniques, we show the following theorems:

\begin{thm}\label{thm 1}

Let $a>1$, $d>1$ be integers such that $d$ is odd and $(a,d)=1$.  Let $\mathcal{P}_{a,d}$ be an arithmetic progression of difference $d$ and let $f$ be a multiplicative function such that $f(a) \neq 0$. Then we have that
    
    \begin{equation} \label{split add}
    f(u+v)=f(u)+f(v) \hspace{1cm}  \text{for all   }  u,v \in \mathcal{P}_{a,d}%
    \end{equation}
    if and only if for all $n$ such that $(n,d)=1$,
    
    \begin{equation} \label{def f odd}
    f(n) = n\chi_d(n)     
    \end{equation}
    for some Dirichlet character modulo $d$ such that $\chi_d(2)=1$.\\
 
   \noindent For $n$ such that $(n,d) \neq 1$, then $f(n)$ takes arbitrary values subject to multiplicativity of $f$.

\end{thm}

The case of progressions with $d$ even is a little more subtle and we need to distinguish between the case $f(a) \neq f(a+d)$ and the case $f(a) = f(a+d)$.

\begin{thm}\label{thm even}
Let $a>1$, $d>1$ be  positive integers such that $d$ is even and $(a,d)=1$. Let  $\mathcal{P}_{a,d}=\{a+kd:k\in\mathbb{Z}\}$ be the arithmetic progression of $a$ modulo $d$ and let $f$ be a multiplicative function such that $f(a) \neq 0$. Then

      \begin{equation*}
    f(u+v)=f(u)+f(v) \hspace{1cm}  \text{for all   }  u,v \in \mathcal{P}_{a,d}%
    \end{equation*}
 if and only if
 
 \begin{enumerate}
     \item  $d \equiv 0 \mod 4$, $f(2)=2$ and  there exists a Dirichlet character $\chi_{d/2}$ modulo $d/2$ such that
\begin{equation}\label{def f d=0 a not =a+d}
        f(n) = n\chi_{d/2}(n) \qquad  \text{ if  } f(a) \neq f(a+d) 
\end{equation}
   \textbf{ or}
   \begin{equation}\label{def f d=0 a=a+d}
         f(n) = \chi_{d/2}(n) \qquad \text{ if  } f(a) = f(a+d) 
   \end{equation}

 whenever $(n,d)=1$;

    \item $d\equiv 2 \mod 4$, and there exists a Dirichlet character $\chi_{d/2}$ modulo $d/2$ such that
    
\begin{align}\label{def f d=2 a not=a+d}
    f(n) &= n\chi_{d/2}(n) \qquad\text{when   } (n,d) =1\\
    \text{and} \hspace{.5cm} 
   \label{2 gamma} f(2^{\gamma}) &= 2^{\gamma}\chi_{d/2}(2^{\gamma-1}) \hspace{1.5cm} \forall \gamma\geq 1
\end{align}
if $f(a) \neq f(a+d)$,

\textbf{or} 

\begin{align}\label{def f d=2 a=a+d}
    f(n) &= \chi_{d/2}(n) \hspace{1.5cm} \text{when   } (n,d) =1\\
    \text{and}\hspace{.5cm}
    \label{2 gamma a=a+d} f(2^{\gamma}) &= 2\chi_{d/2}(2^{\gamma-1}) \qquad \forall \gamma\geq 1 
\end{align}
if $f(a) = f(a+d)$.
 \end{enumerate}

\end{thm}
As a direct corollary of Theorems \ref{thm 1} and \ref{thm even} we get
\begin{corollary}

     Let $f$ be a multiplicative function and suppose there exist two arithmetic progressions $\mathcal{P}_{a,d_1}$ and  $\mathcal{P}_{b,d_2}$, where $(d_1, d_2)=1$ and such that
      \begin{equation*}
    f(u_i+v_i)=f(u_i)+f(v_i) \hspace{1cm}  \text{for all   }  u_i,v_i \in \mathcal{P}_{a,d_i}
    \end{equation*}
for $i=1,2$, then $f$ is the identity function $f(n)=n$.

    \end{corollary}
\section{Playing the game } 

\subsection{Proving that $f$ has the additive property }
We first show that if a function $f$ is defined as in Theorem \ref{thm 1} and Theorem \ref{thm even}, then the additivity condition (\ref{split add}) holds.\\

 Let $u=a+k_1d$ and $v=a+k_2d$ be in $\mathcal{P}_{a,d}$ so that  $n= u+v =2a+kd$ for $k_1+k_2=k$. Let also $\chi$ be a Dirichlet character, where $\chi = \chi_d$ if $d$ is odd and $\chi =\chi_{d/2}$ if $d$ is even. 
 If $f(a)\neq f(a+d)$, then since  by hypothesis we have $(a,d)=1$,  we can infer that for any $k\in \mathbb{Z}$, $(a+kd,d)=1$, so that for any $u$ and $v\in\mathcal{P}_{a,d}$
\begin{align*}
    f(u)+f(v) &= f(a+k_1d) + f(a+k_2d)\\
    & = (a+k_1d)\chi(a+k_1d)+ (a+k_2d)\chi(a+k_2d)\\
    & = \chi(a)(a+k_1d + a+k_2d) = (2a+kd)\chi(a).
\end{align*}
 Hence, we need to show that $f(2a+kd)=(2a+kd)\chi(a)$.\\

  First, consider the case where $d$ is an odd integer and let $f$ be defined by equation \ref{def f odd}.  As $d$ is odd, we also have  $(2a+kd,d)=1$, so that $f(n)=n\chi(n)$ and $\chi(2)=1$ hold by assumption. That is, 
    \begin{align*}
      f(2a+kd)&=(2a+kd)\chi(2a+kd)\\
      &= (2a+kd)\chi(2a)\\
      &= (2a+kd)\chi(2)\chi(a)\\
      &=(2a+kd)\chi(a),
    \end{align*}
    as desired and this concludes the proof of the first direction of Theorem \ref{thm 1} .

  Now suppose that $d$ is even and $f(a)\neq f(a+d)$. We split into two cases. First suppose that $d\equiv 0 \mod 4$, then again, since $(a,d)=1$, we have $(a+kd/2, d) =1$ and also $(2, a+kd/2)=1$, therefore by  (\ref{def f d=0 a not =a+d}),
  \begin{align*}
      f(2a+kd) &= f(2)f(a+kd/2)\\
      &= 2(a+kd/2)\chi(a+kd/2)\\
      & = (2a+kd)\chi(a),
  \end{align*}
  as desired.

  On the other hand, suppose that $d\equiv 2 \mod 4$, and write $2a+kd = 2^{\gamma}c$, where $c$ is an odd integer. Then by the definitions (\ref{def f d=2 a not=a+d}) and (\ref{2 gamma}),
  \begin{align*}
      f(2a+kd) &=f(2^{\gamma}c)\\
      &= f(2^{\gamma})f(c)\\
      & = 2^{\gamma}\chi(2^{\gamma-1})c\chi(c)\\
      &= (2a+kd)\chi(a+kd/2)\\
      &= (2a+kd)\chi(a),
  \end{align*}
  and this finishes the cases $f(a)\neq f(a+d)$.

To end the proof of the first direction of Theorem \ref{thm even}, we are left to prove the case when $f(a)=f(a+d)$ and $d$ is even.
That is, suppose that $f(a) = f(a+d)$, so that by (\ref{def f d=0 a=a+d}) and (\ref{def f d=2 a=a+d}), $f(n) = \chi(n)$ whenever $(n,d) = 1$. Then, since $(a+kd, d) = 1$, we have that

 \begin{align*}
     f(u) + f(v) &= \chi(a+k_1d) + \chi(a+k_2d)\\
     & = \chi(a) + \chi(a) \\
     &= 2\chi(a)
 \end{align*}
by the periodicity of characters, and we therefore have to show that \[f(2a+kd) = 2\chi(a).\]

\noindent In case $d\equiv 2 \mod 4$, write $2a+kd=2^{\gamma}c$, where $c$ is odd, then by (\ref{def f d=2 a=a+d}) and (\ref{2 gamma a=a+d}),
\begin{align*}
    f(2a+kd) &= f(2^{\gamma})f(c)\\
    & =2\chi(2^{\gamma-1})\chi(c)\\
    & =2\chi(2^{\gamma-1}c)\\
    & = 2\chi(a+kd/2)\\
    &= 2\chi(a),
\end{align*}
as needed.
Finally, if $d\equiv 0 \mod 4$, then $d/2$ is even and $a$ must be odd, so that
\begin{align*}
    f(2a+kd) &= f(2)f\left(a + kd/2\right)\\
    &= 2\chi\left(a + kd/2\right)\\
    & =2\chi(a).\\
\end{align*}
This ends the proof of the first direction of Theorems \ref{thm 1} and \ref{thm even}.
   
\subsection{Proving that $f$ is almost the identity}
We now show that if a multiplicative function respects the multiplicative and additive conditions stated in Theorem \ref{thm 1}, then it must be the identity function twisted with a Dirichlet character for all integers $n$ coprime to $d$.

For convenience, we present the necessary assumptions for the proof of the theorems.

\begin{hyp}\label{hyp}
    Given $a>1$ and $d>1$ such that $(a,d)=1$, we let f be a multiplicative function such that $f(a) \neq 0$ and

    \begin{equation*}
        f(u+v)=f(u)+f(v) \hspace{1cm} \forall u,v \in \mathcal{P}_{a,d}
    \end{equation*}
or in the more practical equivalent form

\begin{equation*}\label{split}
    f(2a +kd)= f(a+k_1d)+f(a+k_2d),
\end{equation*}
 for any $k\in \mathbb{Z}$ and any $k_1+k_2=k$.\\ 
 \end{hyp}

The key component of the proof will be to show that $f(1+kd)=1+kd$ (or $f(1+kd)=1$ when f(a) = f(a+d)) for any integer $k$ which will be obtained as Lemma \ref{imp cor} from the following lemma.

    \begin{lemma}\label{lem k factors}
        Under Hypothesis \ref{hyp}, for all $k\in \mathbb{Z}$, we have
        \begin{equation}\label{lemma 1.1}
            f(a+kd)=kf(a+d)-(k-1)f(a)     \\
        \end{equation}
        and
            \begin{equation} \label{lemma 1.2}
                            f(2a+kd)=kf(a+d)-(k-2)f(a). 
            \end{equation}
       
        \begin{proof} 
        We prove the first equation by a simple induction argument on $k$, the case $k=0$ and $k=1$ being trivial. So suppose that $k\geq 1$ and that (\ref{lemma 1.1}) holds for $k-1$. Then on one hand we have  
    \begin{align*}
            f(2a+kd)&=f(a)+f(a+kd),
        \end{align*}
     while on the other hand, we have
          \begin{align*}
            f(2a+kd)&=f(a+d)+f(a+(k-1)d)\\
            &=kf(a+d)-(k-2)f(a),
        \end{align*}
        by induction hypothesis. Putting both together, we get
    \begin{equation*}
            f(a) + f(a+kd)=kf(a+d)-(k-2)f(a)
        \end{equation*}
  from which we obtain 
    \begin{equation*}
           f(a+kd)= kf(a+d)-(k-1)f(a)
        \end{equation*}
        as desired.\\

        Now consider the case of $f(a-kd)$ with $k\geq 0$, then by the same argument we get that 
 \begin{equation*}
           f(a-kd)= kf(a-d)-(k-1)f(a).
        \end{equation*}
Observing that
\begin{align*}
    f(2a+d) &= f(a)+f(a+d)\\
    f(2a+d) &= f(a-d)+f(a+2d),
\end{align*}
we get that 
\begin{align*}
    f(a-d)& = f(a)+f(a+d) - f(a+2d)\\
    &=f(a)+f(a+d) - (2f(a+d) - f(a))
\end{align*}
        using the result we just showed. Now we can do a simple substitution\\ $f(a-d) = 2f(a) -f(a+d)$ from which we obtain that 
        \[f(a-kd) = -kf(a+d) - (-k-1)f(a),\]
       and we can conclude that for all $k\in \mathbb{Z}$
        \[ f(a+kd)= kf(a+d)-(k-1)f(a).\]
        
         The second equation is easily obtained from the first one by adding $f(a)$ on both sides, giving
         \begin{align*}
             f(a)+ f(a+kd)&= kf(a+d)-(k-1)f(a)+ f(a)\\
             f(2a+kd)&= kf(a+d) - (k-2)f(a).
         \end{align*}
        \end{proof}
    \end{lemma}
\begin{corollary}
    Under the same conditions, $f(a+d) \neq 0$ and there exists at most one integer $k_0$ such that $f(a+k_0d)=0$.
    \begin{proof}
        Suppose $f(a+d) = 0$, then, from Lemma \ref{lem k factors} for any $k\geq 1$, we get $f(a+kd) = -(k-1)f(a)$. Choose $k\geq 1$ such that $(a+d, 1+kd)=1$. Then by multiplicativity of $f$ and Lemma \ref{lem k factors},
        \begin{align*}
            f((a+d)(1+kd)) &= f(a + (1+k(a+d))d)\\
            f(a+d)f(1+kd) &= -k(a+d)f(a)\\
            0 &= -k(a+d)f(a),
        \end{align*}
 and therefore we get $f(a) = 0$ which contradicts the assumptions.\\
 
\noindent Moreover, if $f(a+k_0d)=0$ for some positive integer $k_0$, then 
 \[ 0 = k_0f(a+d) - (k_0-1)f(a)
 \]
 so that
 \[f(a+d) = \frac{k_0-1}{k_0}f(a).\]
Hence we get that for any $k\geq1$ 
\begin{align*}
    f(a+kd) &= kf(a+d)-(k-1)f(a)\\
    & = k\left(\frac{k_0-1}{k_0}\right)f(a) -(k-1)f(a)\\
    & = \left(1-\frac{k}{k_0}\right)f(a),
\end{align*}
which is zero only when $k=k_0$.
    \end{proof}
\end{corollary}

    \begin{lemma}\label{imp cor}
        Under Hypothesis \ref{hyp}, if $f(a)\neq f(a+d)$, then for all $k\in \mathbb{Z},$
        \begin{equation*}
                    f(1+kd)=1+kd
        \end{equation*}
        and
        \begin{equation*}
                    f(2+kd)=2+kd.
        \end{equation*}
        \begin{proof}
           Given $k \in \mathbb{Z}$, let $k_1$ and $c\neq 0$ be integers such that $(a+k_1d, 1+kd) = 1$ and $(a+(k_1+c)d, 1+kd) = 1$. Consider
           \[f\big((a+k_1d)(1+kd)\big) = f(a+(ak+k_1+kk_1d)d),\]
           then using multiplicativity and Lemma \ref{lem k factors}, we have
           \[(k_1f(a+d) - (k_1-1)f(a))f(1+kd) = (ak+k_1+kk_1d)f(a+d) - (ak+k_1+kk_1d-1)f(a).\]
           Similarly
            \[f\Big((a+(k_1+c)d)(1+kd)\Big) = f\Big(a+\big(ak+k_1+c+k(k_1+c)d\big)d\Big)\]
            and
           \small\[((k_1+c)f(a+d) - (k_1+c-1)f(a))f(1+kd) = (ak+k_1+c+k(k_1+c)d)f(a+d) - (ak+k_1+c+k(k_1+c)d-1)f(a),\]
           so that subtracting both equations, we get
           \[f(1+kd)c(f(a+d)-f(a)) = c(1+kd)(f(a+d) - f(a)),\]
           which implies, as $f(a)\neq f(a+d)$, that 
           \[f(1+kd) = 1+kd\]
           as desired. The proof of the second result is completely analogous.
        \end{proof}
    \end{lemma}

With this key ingredient in hand, we can now establish the result.
    
    \begin{prop}\label{main prop}
        Given Hypothesis \ref{hyp}  and with the assumption that $f(a)\neq f(a+d)$, then there exists a character $\chi$ such that for all n coprime to d 
        \begin{equation*}
            f(n)=n\chi(n),
        \end{equation*}
        where the modulo for $\chi$ is $d$ if $d$ is odd, and $d/2$ if $d$ is even.

        \begin{proof}
           First let $d$ be odd and let $n$ be an integer with $(n,d) = 1$, then there exists $s$ such that $ns \equiv 1 \pmod d$. 

          Using Dirichlet's theorem, we now choose any sequence of primes in the arithmetic progression of difference $d$ of $s$,
           \[p_1\equiv p_2 \equiv \cdots \equiv p_{\phi(d)} \equiv s \mod d,\]
           so that $(n,p_i)=1$ and $np_i \equiv 1 \pmod d$ for all $1\leq i\leq \phi(d)$. That is, there exists $k_i\in \mathbb{Z}$ such that $np_i = 1+k_id$ and therefore

        \begin{align*}
            f(np_i) &= f(1+k_id)\\
        f(n)f(p_i) &= 1+k_id\\
        & = np_i
        \end{align*}

        Let $f(n)= c \in \mathbb{C}$, then for all the primes in the sequence we have

        \[f(p_i) = \frac{n p_i}{c}.\]

        On the other hand, $p_1p_2\cdots p_{\phi(d)} \equiv 1 \pmod d$ and therefore for some $\Tilde{k}$
        \begin{align*}
            f(p_1p_2\cdots p_{\phi(d)})& = f(1+\Tilde{k}d)\\
            f(p_1)f(p_2)\cdot f(p_{\phi(d)}) &= 1+\Tilde{k}d \\
            &= p_1p_2\cdots p_{\phi(d)}\\
            \frac{n p_1}{c}\cdots \frac{n p_{\phi(d)}}{c} & = p_1p_2\cdots p_{\phi(d)}\\
             \left(\frac{n}{c}\right)^{\phi(d)}p_1p_2\cdots p_{\phi(d)}  & = p_1p_2\cdots p_{\phi(d)},
        \end{align*}
 which allows us to conclude that $\dfrac{n}{c}$ is a $\phi(d)$-th root of unity $\zeta$ and $f(n)= n\overline{\zeta}$.\\

Now, write $f(n) = n g(n)$ where $g$ takes values in the $\phi(d)$-th roots of unity. Then by the nature of $f$, $g$ is multiplicative and we show that $g$ is periodic with period $d$. Suppose that $m\equiv n \mod d$, then find $p$ such that $pn\equiv pm \equiv 1$  and $(n,p) = (m,p)=1$. Then

\begin{align*}
    f(np) &= f(1+kd)\\
   f(n)f(p) & = 1+kd\\
  ng(n)pg(p) & = np \\
 \implies g(n)g(p)& =1 .
\end{align*}
Similarly,
\begin{align*}
    f(mp) &= mp  \\
  mg(m)pg(p) & = mp \\
 \implies g(m)g(p)& =1, 
\end{align*}
from which we deduce that $g(m) = g(n)$ and that $g$ is periodic modulo $d$.\\

Note that this implies that $g$ is completely multiplicative. Indeed, let $p^j$ be any prime power with $(p,d)=1$ and take a sequence of primes $q_1 \equiv  q_2\equiv \cdots q_j \equiv p \mod d$, so that $q_1q_2\cdots q_j \equiv p^j$.
         Since for any $i\leq j$, 
         \[g(q_i) = g(p),\]
         then
         \begin{align*}
             g(q_1q_2\cdots q_j )&= g (p^j)\\
             g(q_1)g(q_2)\cdots g(q_j) &= g(p^j)\\
             g(p)^j &= g(p^j),
         \end{align*}
         and we get that $g$ is completely multiplicative on integers coprime to $d$. Therefore, if $(n,d)=1$, $g(n)= \chi_d(n)$ for some Dirichlet character modulo $d$ and that 
         \[f(n)= n\chi_d(n).\]

         Now, if $d$ is even, we proceed the same way but taking $ns \equiv 1 \mod d/2$ so that $2ns = 2+kd$ and $f(2np_i) = f(2+kd)=2+kd$. The argument is completely analogous to the odd case and we get
         \[f(n)= n\chi_{d/2}(n).\]
   
        \end{proof}
          \end{prop}

        \begin{corollary}\label{cor 2 gamma}
            Under the Hypothesis \ref{split add} and the condition $f(a+d) \neq f(a)$, if $d$ is odd then 
            \[f(2^{\gamma})=2^{\gamma},\]
            and if $d \equiv 2 \mod 4$, then 
            \[f(2^{\gamma})= 2^{\gamma}\chi_{d/2}(2^{\gamma-1}).\]
            \noindent If $d\equiv 0 \mod 4$, we only have $f(2)=2$. 
            \begin{proof}
                Let $d$ be odd, then $(2,d)=1$ and by Proposition \ref{main prop}, 
                \[f(2) = 2\chi(2).\]
                On the other hand, from Lemma \ref{imp cor} we know that $f(2)=2$, so that if $d$ is odd $\chi(2)=1$ and the result follows from complete multiplicativity of characters.\\

                \noindent Now let $d\equiv 2 \mod 4$, which\ means that $d/2$ is odd and $(2^{\gamma-1} , d/2)=1$. That is, there exists a prime $p\in \mathbb{Z}$ such that $ 2^{\gamma-1}p\equiv 1 \mod d/2$ and so we get  $2^{\gamma}p = 2 + kd$ for some $k \in \mathbb{Z}$. Hence we get
                \begin{align*}
                    f(2^{\gamma})f(p) &= f(2 + kd)\\
                    f(2^{\gamma})p\chi(p) &= 2+kd =2^{\gamma}p\\
                    \implies f(2^{\gamma}) &= 2^{\gamma}\chi^{-1}(p)\\
                    & = 2^{\gamma}\chi(2^{\gamma-1}).
                \end{align*}
                \end{proof}
        \end{corollary}

We now show that if $d$ is odd, then under Hypothesis \ref{hyp} $f(a)\neq f(a+d)$, which will ensure that Proposition $\ref{main prop}$ holds and conclude the proof of Theorem \ref{thm 1}. To do so, we first need to investigate the consequences of the condition $f(a) = f(a+d)$.

\begin{lemma}\label{4 cond}
    Under Hypothesis \ref{hyp}, if  $f(a)=  f(a+d)$, then for all $k\geq 0$,

\begin{enumerate}
    \item $f(a+kd) = f(a)$,
    \item $f(2a+kd) = 2f(a)$
    \item $f(1+kd) =1 $ 
    \item $f(2+kd) = 2$.
\end{enumerate}
    \begin{proof}
        The first and second results follow directly from Lemma \ref{lem k factors} and as a consequence, for any $k\geq 0$ taking $k_1$ such that $(a+ k_1d, 1+kd)=1$, we get that
        \begin{align*}
          f((a+k_1d)(1+kd)) &= f(a+\Tilde{k}d)\\
        f(a+k_1d)f(1+kd) & = f(a+\Tilde{k}d)\\
        f(a)f(1+kd) = f(a)
        \end{align*}
        from which it follows that $f(1+kd)=1$. The proof of the last equality is completely analogous.
    \end{proof}
\end{lemma}

\begin{corollary}
    Under Hypothesis \ref{split add}, if $d$ is odd, then $f(a) \neq f(a+d)$.
    \begin{proof}
        Suppose for contradiction that $f(a) = f(a+d)$. Then for all $k\geq 0$, $f(1+kd)=1$. That is, solving for $k$

        \[1+kd\equiv 2 \mod 4,\]
        there exists an integer $m$ such that  
        \begin{equation}
            1+kd = 2+4m = 2(1+2m). \label{4m}
        \end{equation}
Applying the function on both sides, we get

\begin{align*}
    f(1+kd) &= f(2)f(1+2m),\\
    1 &= 2f(1+2m),
\end{align*}
so that we deduce that $f(1+2m)=\frac{1}{2}$. Now multiplying equation (\ref{4m}) by 2 and using $f$ again, we get
\begin{align*}
    f(2+2kd) &= f(4)f(1+2m),\\
    2 &= f(4)\frac{1}{2},
\end{align*}
And therefore we deduce that $f(4) = 4.$
Now, proceeding by induction, we repeat the same argument by replacing 2 with $2^m$ and 4 with $2^{m+1}$ which leads to 
\[f(2^m) = 2^m.\]
However, there exists $1\leq \gamma\leq d-1$ such that $2^{\gamma} \equiv 1 \mod d$ and so for some $\gamma > 1$,

\[
    f(2^{\gamma}) = f(1+kd),
\]
from which we would get that $2^{\gamma} = 1$, which is obviously not true. We can therefore conclude that if $d$ is odd, $f(a) \neq f(a+d)$. 
    \end{proof}
\end{corollary}

This concludes the proof of Theorem \ref{thm 1}, and part of the proof of Theorem \ref{thm even} in the case $f(a)\neq f(a+d)$ and $(n,d)=1$. It remains to prove the cases when $d$ is even and $f(a)=f(a+d)$.

\begin{prop}
    Let $d$ be even, then under Hypothesis \ref{hyp} and the assumption that $f(a) = f(a+d)$, there exists a Dirichlet character modulo $d/2$ such that
    \[f(n) = \chi_{d/2}(n) \qquad \text{ if } (n,d)=1.\]
    \begin{proof}
       The proof follows the lines of the proof of Proposition \ref{main prop}, with the key difference that we use Lemma \ref{4 cond}, namely the fact that for any $k\in \mathbb{Z}$, $f(1+kd) =1$ and $f(2+kd)=2$.

  Let $m$ and $n$ be integers such that $(n,d) = (m,d)=1$ and $m\equiv n \mod d/2$. Find an odd prime $p$ such that $p\equiv m^{-1} \mod d/2$ and $(p,m) = (p,n) =1$. Then we have  $pm \equiv pn \equiv 1 \mod d/2$, and therefore for some integers $k_1$ and $k_2$,

    \begin{align*}
        f(2pm) &= f(2+k_1d)\\
        f(2)f(p)f(m) &= 2,
    \end{align*}
and similarly   
    \begin{align*}
        f(2pn) &= f(2+k_2d)\\
        f(2)f(p)f(n) &= 2,
    \end{align*}
    from which we deduce that $f(n) = f(m)$ and we conclude that $f$ is periodic modulo $d/2$. As in the proof of Proposition \ref{main prop}, we can easily show that in this case $f$ is completely multiplicative on integers coprime to $d$ so that we indeed get that $f(n)$ agrees with a Dirichlet character modulo $d/2$ on all integers $n$ such that $(n,d)=1$.

\end{proof}
\end{prop}

\begin{corollary}
    Under the same conditions, if $d\equiv 2 \mod 4$, then 
    \[f(2^{\gamma}) = 2\chi_{d/2}(2^{\gamma}-1).\]
    \begin{proof}
        The proof is analogous to the one of Corollary \ref{cor 2 gamma}, we leave it to the reader to verify.
    \end{proof}
\end{corollary}

This ends the proof of Theorem \ref{thm even} and concludes the characterization of multiplicative functions respecting the given additive property on arithmetic progressions.

\bibliographystyle{plain}
\bibliography{biblio.bib}

\end{document}